\documentclass{amsart}

\usepackage{amsthm,amssymb,latexsym,amsmath,amsfonts,amscd,amstext}
\usepackage[all]{xy}
\usepackage[english]{babel}
\usepackage[utf8x]{inputenc} \usepackage[T1]{fontenc}
\usepackage{color}
\usepackage{graphicx}	
\usepackage{mathrsfs}
\usepackage[all]{xy}

\usepackage{latexsym}
\usepackage{epsfig}
\usepackage[mathscr]{eucal}

\usepackage{tikz}
\usepackage{tikz-cd}
\usetikzlibrary{arrows,automata}
\usetikzlibrary{fit}
\usepackage{wrapfig}

\newcommand{\p}[1]{{\mathbb{P}^{#1}}}
\newcommand{\calf}{{\mathcal F}}

\newcommand{\caln}{{\mathcal N}}

\newcommand{\C}{\mathbb{C}}

\newtheorem{thm}{Theorem}

\newtheorem{lem}[thm]{Lemma}
\newtheorem{pps}[thm]{Proposition}
\newtheorem{dfn}[thm]{Definition}

\title{Holomorphic pre-symplectic form on the nested Hilbert scheme ${\rm Hilb}^{3,4}(\C^2)$  }

\author{Rodrigo A. von Flach}
\address{Centro de Ci\^encias Exatas e Tecnol\'ogicas \\
Universidade Federal do Rec\^oncavo da Bahia \\
Rua Rui Barbosa, 710 \\
44380-000 Cruz das Almas, BA, Brazil}
\email{rodrigovonflach@ufrb.edu.br}

\author{Newiton Braga Neto}
\address{Centro de Ci\^encias Exatas e Tecnol\'ogicas \\
	Universidade Federal do Rec\^oncavo da Bahia \\
	Rua Rui Barbosa, 710 \\
	44380-000 Cruz das Almas, BA, Brazil}
\email{newitonbraga@gmail.com}
	
\begin{document}
 
\begin{abstract}
We regard the pre-hyperk\"ahler structure on the moduli space of framed flags of sheaves $\mathcal{F}(1,3,1)$ on the projective plane $\mathbb{P}^2$ via an adaptation of the ADHM construction of framed sheaves. Then, we study and categorize the degenerate points of the holomorphic pre-symplectic form presented in the moduli space.

\end{abstract}

\maketitle

\tableofcontents

\section{Introduction}

Moduli spaces of flags of sheaves has been playing an important role since the work of Grojnowski \cite{G} and Nakajima \cite{nakajima}. In 2000, Baranovsky \cite{B} constructed an action of the Heisenberg algebra in the cohomology of moduli spaces of sheaves on surfaces giving a higher rank generalization. A few years later, Bruzzo et al. \cite{A01-01} present flags of sheaves as a tool to study supersymmetric quantum mechanical model in string theory. Flags of sheaves also appeared in the work of Chuang et al. \cite{CDDT} in order to give a string theoretic derivation for the conjecture of Hausel, Letellier and Rodriguez-Villegas on the cohomology of character varieties with marked points. More recently, von Flach and Jardim \cite{vFJ} presented a detailed account of the ADHM construction of the moduli space of framed flags of sheaves on the projective plane. 

From a geometric point of view, for specific topological invariants, this moduli space has a holomorphic structure defined as holomorphic pre-symplectic form. More precisely, let $(E,F,\varphi)$ be triples consisting of a torsion free sheaf $F$ on $\p2$, a framing $\varphi$ of $F$ at a line $\ell_\infty$ and a subsheaf $E$ of $F$ such that the quotient $F/E$ is supported away from the framing line $\ell_\infty$. We denote by $\calf(r,n,l)$ the moduli space of such triples where $r:={\rm rk}(E)={\rm rk}(F)$, $n:=c_2(E)$, and $l:=h^0(F/E)$ are fixed invariants.  

It was proved in \cite{vFJ} that $\calf(r,n,1)$ is an irreducible, nonsingular quasi-projective variety of dimension $2rn+r+1$ by using the same techniques as in \cite[Section 3]{A01-01}. Note also that $\calf(1,n,l)$ coincides with the nested Hilbert scheme ${\rm Hilb}^{n,n+l}(\C^2)$ of points in $\C^2$. We are particularly interested in the case $l=1$ because $\calf(1,n,1) =  {\rm Hilb}^{n,n+1}(\C^2)$, which is known to be smooth. Furthermore in \cite{vFJ} it is proved that $\calf(1,n,1)$ admits the structure of a holomorphic pre-symplectic manifold, that is, $\calf(1,n,1)$ is a K\"ahler manifold equipped with a natural closed holomorphic 2-form $\Omega$. In the simplest possible case, namely $n=1$, $\Omega$ is generically non-degenerate.

In this paper we study the degenerate points of $\Omega$ on $\calf(1,3,1)$ continuing the study initiated in \cite{vFJ}[Section 7.3]. These points were there characterized by analyzing the matrices that describe the enhanced ADHM variety $\caln(1,2,1)$ associated with $\calf(1,2,1)$. In other words, it was proved in \cite{vFJ} that the moduli space of framed stable representations of the \emph{enhanced	ADHM quiver}
\begin{equation}
\label{enhanced-ADHM-quiver}
\begin{tikzpicture}[->,>=stealth',shorten >=1pt,auto,node distance=2.5cm,
semithick]
\tikzstyle{every state}=[fill=white,draw=none,text=black]

\node[state]    (A)                {$e_1$};
\node[state]    (B) [left of=A]    {$e_2$};
\node[state]    (C) [right of=A]   {$e_{\infty}$};

\path (A) edge [loop above]          node {$\alpha$}     (A)
edge [loop below]          node {$\beta$}      (A)
edge [out=20,in=160]       node {$\eta$}       (C)
edge [out=-160,in=-20]     node {$\gamma$}     (B)
(B) edge [loop above]          node {$\alpha'$}    (B)
edge [loop below]          node {$\beta'$}     (B)
edge [out=20,in=160]       node {$\phi$}       (A)
(C) edge [out=-160,in=-20]     node {$\xi$}        (A);
\end{tikzpicture}
\end{equation}
with the relations
\begin{equation}
\begin{array}{rccc}
\label{eq:enhADHMrel02}
\alpha'\beta'-\beta'\alpha'& \alpha\beta - \beta\alpha + \xi\eta, & \alpha\phi-\phi\alpha', & \beta\phi-\phi\beta',\\
\eta\phi, & \begin{array}{lr}
\gamma\xi, & \phi\gamma, \end{array} & \gamma\alpha-\alpha'\gamma, & \gamma\beta - \beta'\gamma,
\end{array}
\end{equation}
and dimension vector $(r,c,c')$ is isomorphic to $\calf(r,c-c',c')$. All properties of the former space are obtained by analyzing the moduli space of stable quiver representations.

%As it is checked in Section \ref{e-adhm-q}, the moduli space of stable representations of the enhanced ADHM quiver coincide with the \emph{Hecke correspondences} for the one-loop quiver, as introduced by Nakajima in \cite[Section 5]{N1} and \cite[Section 8]{N2}. Although the smoothness and irreducibilibity of Hecke correspondences for $l=1$ are well known to specialists, it is hard to find a detailed account. In particular, \cite[Theorem 8.3]{N2} only provides a set-theoretical bijection between the Hecke correspondences and $\calf(r,n,l)$, while we prove that these varieties are isomorphic as schemes.

The paper is outlined as follows. In Section \ref{sec.flags} we present briefly how \cite{vFJ} proved that the enhanced ADHM variety $\mathcal{N}(r,n+l,l)$ and the moduli space of flags of sheaves on $\mathbb{R}$ $\mathcal{F}(r,n,l)$ are isomporphics. The construction of holomorphic pre-symplectic structure on $\mathcal{N}(1,3,1)$ and the study of its degenerate points was performed in Section \ref{geometry}.

%%%%%%%%%%%%%%%%%%%%%%%%%%%%%%%%%%%%%%%%%%%%%%%%%%%%%%%%%%%%%%
%%%%%%%%%%%% framed flags as quiver varieties %%%%%%%%%%%%%%%%
%%%%%%%%%%%%%%%%%%%%%%%%%%%%%%%%%%%%%%%%%%%%%%%%%%%%%%%%%%%%%%

\section{Framed flags of sheaves on $\p2$ as enhanced ADHM varieties} \label{sec.flags}

Fix a line $\ell_\infty\subset\p2$; recall that a \emph{framing} of a coherent sheaf $F$ on $\p2$ at the line $\ell_\infty$ is the choice of an isomorphism $\varphi:F|_{\ell\infty}\to{\mathcal O}_{\ell\infty}^{\oplus r}$, where $r$ is the rank of $F$.
A \emph{framed flag of sheaves} on $\p2$ is a triple $(E,F,\varphi)$ consisting of a torsion free sheaf $F$ on $\p2$, a framing $\varphi$ of $F$ at the line $\ell_\infty$, and a subsheaf $E$ of $F$ such that the quotient $F/E$ is supported away from the framing line $\ell_\infty$. Note that the existence of a framing forces $c_1(F)=0$, while the last condition implies that $c_1(E)=0$, and that $F/E$ must be a 0-dimensional sheaf. Thus the triple $(E,F,\varphi)$ has three numerical invariants: $r:={\rm rk}(E)={\rm rk}(F)$, $n:=c_2(F)$ and $l:=h^0(F/E)$; note that $c_2(E)=n+l$. 

It was proved in \cite{vFJ}[Theorem 18] that the moduli space of flags of sheaves $\calf(r,n,l)$ is isomorphic to the moduli space of stable representations of the enhanced ADHM quiver $\caln(r,n+l,l)$, also called enhanced ADHM variety, defined as follows. Consider the following quiver as the \textit{enhanced ADHM quiver}\index{Enhanced ADHM quiver}
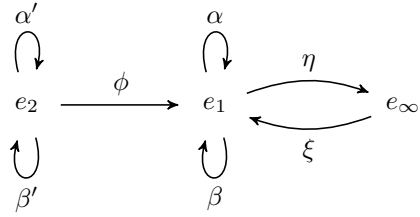
\begin{equation*}\label{eq:quiverADHMaument-est}
\begin{tikzpicture}[->,>=stealth',shorten >=1pt,auto,node distance=2.5cm,
semithick]
\tikzstyle{every state}=[fill=white,draw=none,text=black]

\node[state]    (A)                {$e_1$};
\node[state]    (B) [left of=A]    {$e_2$};
\node[state]    (C) [right of=A]   {$e_{\infty}$};

\path (A) edge [loop above]          node {$\alpha$}     (A)
edge [loop below]          node {$\beta$}      (A)
edge [out=20,in=160]       node {$\eta$}       (C)
%	edge [out=-160,in=-20]     node {$\gamma$}     (B)
(B) edge [loop above]          node {$\alpha'$}    (B)
edge [loop below]          node {$\beta'$}     (B)
edge [out=0,in=180]       node {$\phi$}       (A)
(C) edge [out=-160,in=-20]     node {$\xi$}        (A);
\end{tikzpicture}
\end{equation*}
with ideal generated by relations
\begin{equation*}
\begin{array}{ccccc}
\label{eq:quiverADHMenh}
\alpha\beta - \beta\alpha + \xi\eta, & \alpha\phi -\phi\alpha', & \beta\phi-\phi\beta', &\eta\phi, & \alpha'\beta'-\beta'\alpha'.
\end{array}
\end{equation*}
Then a representation of the quiver above is given by $X=(A,B,I,J,A',B',F)$ such that $A,$ $B\in End(V)$, $I\in Hom(W,V)$, $J\in Hom(V,W)$, $A'$, $B'\in End(V')$ and $F\in Hom(V',V)$, see the diagram below,
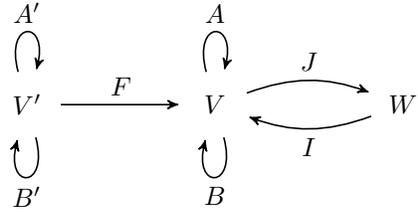
\begin{equation*}
\begin{tikzpicture}[->,>=stealth',shorten >=1pt,auto,node distance=2.5cm,
semithick]
\tikzstyle{every state}=[fill=white,draw=none,text=black]

\node[state]    (A)                {$V$};
\node[state]    (B) [left of=A]    {$V'$};
\node[state]    (C) [right of=A]   {$W$};

\path (A) edge [loop above]     node {$A$}     (A)
edge [loop below]     node {$B$}     (A)
edge [out=20,in=160]      node {$J$}     (C)
%edge [bend left]      node {$\gamma$}     (B)
(B) edge [loop above]     node {$A'$}    (B)
edge [loop below]     node {$B'$}    (B)
edge [out=0, in=180]      node {$F$}     (A)
(C) edge [out=-160,in=-20]      node {$I$}     (A);
\end{tikzpicture}
\end{equation*}
satisfying the equations
\begin{equation}
\begin{array}{ccc}	\label{eq:repquiverADHMenh}
~~ [A,B]+IJ = 0, & JF=0, & ~~ \\ 
~~ [A',B']=0, & AF-FA'=0, & BF-FB'=0,
\end{array}
\end{equation}
which will be also called \textit{enhanced ADHM equations} in this work. This representation is stable if satisfies
	\begin{enumerate}
	\item[$(S.1)$] $F\in Hom(V',V)$ is injective;
	\item[$(S.2)$] The ADHM data $\mathcal{A}=(W,V,A,B,I,J)$ is stable, i.e., there is no proper subspace $0\subset S\subsetneq V$ preserved by $A$, $B$ and containing the image of $I$.
\end{enumerate}

This is a reasonable stability condition since it was proved in \cite{vFJ}[Lemma 4] that these conditions are equivalent to the $\Theta$-semistability condition defined by King in \cite{A01-01-20}. This proof is completely analogous to the proof given by Bruzzo, et al. in \cite{A01-01}[Lemma 3.1]. Then, by using Geometric Invariant Theory techniques, by analogy with \cite{A01-01-20} and \cite[Section 3.2]{A01-01} one can construct the moduli space of framed stable representations of the enhanced ADHM quiver $\caln(r,c,c^{\prime})$, where $\dim(W)=r$, $\dim(V)=c$ and $\dim(V^{\prime})=c^{\prime}$. A detailed construction for this moduli space can be found in \cite{vFJ}[Section 4].

Since it is proved in \cite{vFJ} that $\calf(1,n,l)$ coincides with the nested Hilbert scheme ${\rm Hilb}^{n,n+l}(\C^2)$ of points in $\C^2$, for $l=1$ we have $\calf(1,n,1) =  {\rm Hilb}^{n,n+1}(\C^2) $ which is smooth. Furthermore $\calf(1,n,1)$ admits the structure of a holomorphic pre-symplectic manifold, that is, $\calf(1,n,1)$ is a K\"ahler manifold equipped with a natural closed holomorphic 2-form $\Omega$ (see \cite{vFJ}). 

Considering only the smooth moduli space, i.e., $\caln(1,c,1)$ (see \cite{vFJ}[Section 5]), its tangent space is given by the quotient 
\begin{equation}\label{eq:tangent_space}
T_{X}\caln(1,c,1) = \dfrac{ker(d_1)}{im(d_0)},
\end{equation}

where
\begin{equation}
\label{eq:complexo-cr-c'}
\xymatrix{%\mathcal{C}(X): 
	& \txt{$End(V)$\\ $\oplus$ \\ $End(V')$} \ar[r]^-{d_0} 
	& \txt{$End(V)^{\oplus^2}$\\ $\oplus$ \\ $Hom(W,V)$\\ $\oplus$ \\ $Hom(V,W)$\\ $\oplus$ \\ $End(V')^{\oplus^2}$\\ $\oplus$ \\ $Hom (V',V)$}  \ar[r]^-{d_1}
	& \txt{$End(V)$\\ $\oplus$ \\ $Hom(V',V)^{\oplus^2}$\\ $\oplus$ \\ $Hom(V',W)$\\ $\oplus$ \\ $End(V')$}% \ar[r]^-{d_2}
%	& \txt{$Hom(V',V)$}
}
\end{equation}
is given by \begin{equation*}
\begin{array}{rcl}
d_0(h,h')                & = & ([h,A], [h,B], hI, -Jh, [h',A'], [h',B'], hF-Fh')\\
d_1(a,b,i,j,a',b',f)     & = & ([a,B] + [A,b] + Ij + iJ, Af+aF-Fa'-fA', \\
&   & Bf+bF-Fb'-fB', jF +Jf, [a',B']+[A',b']).
\end{array}
\end{equation*}

%%%%%%%%%%%%%%%%%%%%%%%%%%%%%%%%%%%%%%%%%%%%%%%%%%%%%%%%%%%%%%
%%%%%%%%%%%%%%%%%%%  geometric structures  %%%%%%%%%%%%%%%%%%% 
%%%%%%%%%%%%%%%%%%%%%%%%%%%%%%%%%%%%%%%%%%%%%%%%%%%%%%%%%%%%%%

\section{Geometric structures on $\mathcal{N}(1,3,1)$} \label{geometry}

The goal of this section is to study geometric structures on the moduli space of framed flags of sheaves, motivated by the fact that the moduli space of framed torsion free sheaves on $\p2$ is known to be a hyperk\"ahler manifold. We will fix dimension vector $(1,3,1)$ in order to study the degenerate points of the holomorphic pre-symplectic form $\Omega$. The pre-hyperk\"ahler structure was defined in \cite{vFJ}[Section 7] and the construction of this structure is presented for $\caln(r,c,1)$ for the sake of completeness. 

Recall that a \emph{hyperk\"ahler manifold} is a Riemannian manifold $(M,g)$ equipped with three parallel complex structures $(\Gamma_1,\Gamma_2,\Gamma_3)$ satisfying the usual quaternionic relations; in addition, each 2-form $\omega_k(\cdot,\cdot):=g(\Gamma_k\cdot,\cdot)$ is a K\"ahler form for the K\"ahler manifold $(M,g,\Gamma_k)$. One can then define a symplectic form $\Omega:=\omega_2+i\omega_3$, which is holomorphic with respect to the complex structure $\Gamma_1$; the triple $(M,\Gamma_1,\Omega)$ is called the \emph{holomorphic symplectic manifold} associated with the hyperk\"ahler ma\-ni\-fold $(M,g,\Gamma_1,\Gamma_2,\Gamma_3)$.

\begin{dfn}
A \emph{pre-hyperk\"ahler manifold} is a K\"ahler ma\-ni\-fold $(M,g,\Gamma)$ equipped with a pair of closed 2-forms $(\omega_1,\omega_2)$ satisfying
\begin{equation}\label{pre-hk}
\omega_2(\cdot,\cdot) = \omega_3(\cdot,\Gamma\cdot) .
\end{equation}
\end{dfn}

Given a pre-hyperk\"ahler manifold $(M,g,\Gamma,\omega_2,\omega_3)$, one can define the closed 2-form $\Omega:=\omega_2+i\omega_3$; condition (\ref{pre-hk}) implies that
$$ \Omega(\cdot,\Gamma\cdot) = i\Omega(\cdot,\cdot) $$
hence $\Omega$ is holomorphic with respect to $\Gamma$. This observation motivates the following definition.

\begin{dfn}
A \emph{holomorphic pre-symplectic manifold} is a triple $(M,\Gamma,\Omega)$ consisting of a complex manifold $(M,\Gamma)$ equipped with a holomorphic pre-symplectic structure $\Omega$.
\end{dfn}

The holomorphic pre-symplectic manifold $(M,\Gamma,\Omega)$ decribed in the paragraph before the previous definition is called the holomorphic pre-symplectic manifold associated with the pre-hyperk\"ahler manifold $(M,g,\Gamma,\omega_2,\omega_3)$. Note that $\Omega$ is non-degenerate if and only if both $\omega_2$ and $\omega_3$ are non-degenerate.

It was proved in \cite{vFJ}[Section 7.1]  that $\calf(1,n,1)$ admits the structure of a pre-hyperk\"ahler manifold; this is done by embedding it into a hyperk\"ahler manifold.

%-------------------------------------------

%-------------------------------------------

\subsection{The pre-hyperk\"ahler structure on $\caln(1,c,1)$}

\indent In this section, one can find the consequences of the fact that the moduli space $\mathcal{N}(1,c,1)$ is a subvariety of the hyperk\"ahler manifold $\mathcal{W}(1,c,1)=(\mathcal{W}(1,c,1),\langle\mbox{ , }\rangle, \Gamma_1,\Gamma_2,\Gamma_3)$. It was proved in \cite{vFJ}[Section 7.1] that this is true for the general case $\mathcal{N}(r,c,c')$. However, here it is fixed the moduli space of framed stable representations of the ADHM quiver of numerical type $(1,c,1)$, because this is the only case in which the variety is smooth. First, note that there exists the inclusion map
\begin{equation*}
\begin{tikzcd}
\mathcal{N}(1,c,1)\arrow[hookrightarrow]{r}{\iota} & (\mathcal{W}(1,c,1),\langle\mbox{ , }\rangle, \Gamma_1,\Gamma_2,\Gamma_3).
\end{tikzcd}\end{equation*}

Hence, associated with this inclusion, there exists a complex structure on $\mathcal{N}(1,c,1)$ inherited by the pull-back, $\iota^*\Gamma_1$, and a closed degenerate $2$-form $\Omega=\iota^*\omega_2+\sqrt{-1}\iota^*\omega_3$. Indeed, let $(a,b,i,j,a',b',f,0)\in\mathcal{N}(1,c,1)$. Thus,
\begin{align*}
\iota^*\Gamma_1(a,b,i,j,a',b',f,0) =& \Gamma_1(\iota_*a,\iota_*b,\iota_*i,\iota_*j,\iota_*a',\iota_*b',\iota_*f,0)\nonumber \\
=& (\sqrt{-1}a,\sqrt{-1}b,\sqrt{-1}i,\sqrt{-1}j,\sqrt{-1}a',\sqrt{-1}b',\sqrt{-1}f,0)
\end{align*}
is clearly a complex structure on $\mathcal{N}$. Moreover, let $x_1=(a_1,b_1,i_1,j_1,a'_1,b_1',f_1,0)$ and $x_2=(a_2,b_2,i_2,j_2,a'_2,b_2',f_2,0)$ in $\mathcal{N}$. It is easy to check that $(\mathcal{N}, \iota^*\langle\mbox{ , }\rangle, \iota^*\Gamma_1 )$ has a K\"ahler structure. The $2$-form $\Omega$ is given by
\begin{align*}
\Omega(x_1,x_2) =& (\iota^*\omega_2+\sqrt{-1}\iota^*\omega_3)(x_1,x_2)  \nonumber\\
=& (\omega_2+\sqrt{-1}\omega_3)(\iota_*x_1,\iota_*x_2) \nonumber \\ 
%=& (\omega_2+\sqrt{-1}\omega_3)((a_1,b_1,i_1,j_1,a'_1,b_1',f_1,0),(a_2,b_2,i_2,j_2,a'_2,b_2',f_2,0))\nonumber \\
%=& -\frac{1}{2}tr(a_2b_1 - b_2a_1 + b_1^{\dagger}a_2^{\dagger}-a_1^{\dagger}b_2^{\dagger} + i_2j_1-j_2i_1+j_1^{\dagger}i_2^{\dagger}-i_1^{\dagger}j_2^{\dagger}\nonumber \\
%&+b_1^{\prime\dagger}a_2^{\prime\dagger}-a_1^{\prime\dagger}b_2^{\prime\dagger} +a'_2b'_1 - b_2'a'_1)-\nonumber \\
%&-\frac{(\sqrt{-1})^{2}}{2}tr(-a_2b_1 + b_2a_1 + b_1^{\dagger}a_2^{\dagger}-a_1^{\dagger}b_2^{\dagger} - i_2j_1+j_2i_1+j_1^{\dagger}i_2^{\dagger}-i_1^{\dagger}j_2^{\dagger}\nonumber \\
%&+b_1^{\prime\dagger}a_2^{\prime\dagger}-a_1^{\prime\dagger}b_2^{\prime\dagger} -a'_2b'_1 + b_2'a'_1)\nonumber \\
%=& \frac{1}{2}tr(-a_2b_1 + b_2a_1 - b_1^{\dagger}a_2^{\dagger}+a_1^{\dagger}b_2^{\dagger} - i_2j_1 + j_2i_1 - j_1^{\dagger}i_2^{\dagger} + i_1^{\dagger}j_2^{\dagger}\nonumber \\
%& -b_1^{\prime\dagger}a_2^{\prime\dagger}+a_1^{\prime\dagger}b_2^{\prime\dagger} -a'_2b'_1 + b_2'a'_1 +b_1^{\prime\dagger}a_2^{\prime\dagger}-a_1^{\prime\dagger}b_2^{\prime\dagger} -a'_2b'_1 + b_2'a'_1 -\nonumber \\
%& -a_2b_1 + b_2a_1 + b_1^{\dagger}a_2^{\dagger}-a_1^{\dagger}b_2^{\dagger} - i_2j_1+j_2i_1+j_1^{\dagger}i_2^{\dagger}-i_1^{\dagger}j_2^{\dagger})\nonumber \\
&= tr(-a_2b_1+b_2a_1-i_2j_1+i_1j_2-a'_2b'_1+b'_2a'_1).
\end{align*} 

Note that by taking $u=(0,0,0,0,0,0,f,0)\in T\mathcal{N}$, $\Omega_{X}(u,v)\equiv 0$ for all $v\in T\mathcal{N}$, i.e., $\Omega$ is in fact a degenerate $2$-form. Also, it is easy to check that the $2$-forms $\iota^*\omega_2$ and $\iota^*\omega_3$ satisfy
\begin{equation*}
\left\{\begin{array}{lll}
\iota^*\omega_2(u,v) & = & \iota^*\omega_3(u,\Gamma_1v)\\
\iota^*\omega_3(u,v) & = & -\iota^*\omega_2(u,\Gamma_1v)
\end{array}\right. .
\end{equation*}
In other words, $\mathcal{N}(1,c,1)$ admits the structure of a pre-hyperk\"ahler manifold.

%-------------------------------------------

\subsection{Degenerate points of the holomorphic pre-symplectic form on $\caln(1,3,1)$}

We consider now the case, $c=3$ to precisely determine the degeneration locus of the 
closed holomorphic $2$-form $\Omega$ defined above, that is for which points $X\in\mathcal{N}(1,3,1)$ the linear map
\begin{equation*} \begin{array}{lcl}
T_X\mathcal{N}(1,3,1) & \longrightarrow & \left( T_X\mathcal{N}(1,3,1) \right)^* \\
             u        & \longmapsto     & \Omega_X(u,\cdot)
\end{array}
\end{equation*}
fails to be an isomorphism.

First, we need to prove the following auxiliary Lemma.

\begin{lem}\label{lemma-auxiliar}
Let $X=(W,V,V',A,B,I,J,A',B',F,G)$ be a framed stable representation of the enhanced ADHM quiver of numerical type $(1,3,1)$. Thus, there exists a change of basis for $V$ such that
\begin{itemize}
\item[(i)] $A=\begin{bmatrix}A' & 0 & 0\\ 0 & A_2 & 0\\ 0 & 0 & A_3 \end{bmatrix}$, $B=\begin{bmatrix}B' & 0 & 0 \\ 0 & B_2 & 0 \\ 0 & 0 & B_3\end{bmatrix}$, $F=\begin{bmatrix} 1\\ 0 \\ 0\end{bmatrix}$, if $A$ and $B$ are diagonalizable;
\item[(ii)] If $A$ or $B$ are not diagonalizable, we have the 3 cases below to analyze. We will consider $B$ not diagonalizable to fix ideias.
\item[(ii.1)] $A=\begin{bmatrix}A' & A_{12} & A_{13} \\ 0 & A' & A_{12} \\ 0 & 0 & A\end{bmatrix}$, 
$B=\begin{bmatrix}B' & 1 & 0 \\ 0 & B' & 0 \\ 0 & 0 & B_3\end{bmatrix}$, $F=\begin{bmatrix} 1\\ 0 \\ 0 \end{bmatrix}$;
\item[(ii.2)] $A=\begin{bmatrix}A' & A_{12} & A_{13} \\ 0 & A' & 0 \\ 0 & 0 & A\end{bmatrix}$, 
$B=\begin{bmatrix}B' & 1 & 0 \\ 0 & B' & 1 \\ 0 & 0 & B'\end{bmatrix}$, $F=\begin{bmatrix} 1\\ 0 \\ 0 \end{bmatrix}$;
\item[(ii.3)] $A=\begin{bmatrix}A' & 0 & 0 \\ 0 & A_2 & A_{23} \\ 0 & 0 & A_2\end{bmatrix}$, 
$B=\begin{bmatrix}B' & 0 & 0 \\ 0 & B_2 & 1 \\ 0 & 0 & B_2\end{bmatrix}$, $F=\begin{bmatrix} 1\\ 0 \\ 0 \end{bmatrix}$;
\end{itemize}
\end{lem}

The proof of this lemma is analogous to the one of \cite{vFJ}[Lemma 26], as one can check by using equations \eqref{eq:repquiverADHMenh}.
 
 We are finally in a position to prove the main result of this section.

\begin{pps}\label{prop28}
Let $\mathcal{N}(1,3,1)$ be the moduli space of framed stable representations of the enhanced ADHM quiver of numerical type $(1,3,1)$. Fix a framed stable representation $X=(A,B,I,J,A',B',F)$. Then the $2$-form $\Omega_X$ defined on $T_X\mathcal{N}(1,3,1)$ is non-degenerate if and only if there is a change of basis for $V$ such that the matrices associated with the endomorphisms $A$ and $B$ are diagonalizable or $B$ is not diagonalizable and its Jordan normal form is given by $B=\begin{bmatrix}B' & 0 & 0 \\ 0 & B_2 & 1 \\ 0 & 0 & B_2\end{bmatrix}$. 
\end{pps}

\begin{proof}
According to Lemma \ref{lemma-auxiliar}, there is a change of basis for $V$ such that the matrices $A$, $B$ and $F$ are given by (i), (ii.1), (ii.2), (ii.3) and the proof consists in verifying if the holomorphic form is non-degenerated for each case. We will present the analysis only of the case (i). The other cases are analogous as one can check through tedious computations. 

Recall that if $r=1$, then the map $J\in Hom(V,W)$ must vanish, since $X$ is stable (see \cite[Proposition 2.8]{nakajima}), and recall that if $c'=1$, then $[A',B']=0$, for all $A',$ $B'\in V'$. Thus, the enhanced ADHM equations reduce to
\begin{align*}
 [A,B] =0,\quad AF-FA'=0,\quad BF-FB'=0.
\end{align*} 
 	
 	Suppose that $A$ and $B$ are diagonalizable. Thus, it follows from Lemma \ref{lemma-auxiliar} (i) that there exists a change of basis for $V$ such that
 	%\begin{equation}
 	%A=	\begin{bmatrix}
 	%		A' & 0 \\ 
 	%		0 & a_2
 	%	\end{bmatrix},\quad 
 	%B=	\begin{bmatrix}
 	%	B' & 0 \\
 	%	 0 & b_2
 	%	 \end{bmatrix},\quad 
 	%F=	\begin{bmatrix} 
 	%		1\\ 
 	%		0
 	%	\end{bmatrix}
 	%\end{equation} 
 	%\begin{equation}
 	%
 	%F =\left[\begin{array}{c}
 	%1\\
 	%0
 	%\end{array}\right].
 	%\end{equation}
 $$A=\begin{bmatrix}A' & 0 & 0\\ 0 & A_2 & 0\\ 0 & 0 & A_3 \end{bmatrix}, \quad B=\begin{bmatrix}B' & 0 & 0 \\ 0 & B_2 & 0 \\ 0 & 0 & B_3\end{bmatrix}, \quad F=\begin{bmatrix} 1\\ 0 \\ 0\end{bmatrix}.$$ 
 	%A=\left[\begin{array}{cc}
 	% \lambda_a & 1\\
 	%0 & \lambda_a\end{array}\right], \quad B=\left[\begin{array}{cc}
 	% \lambda_b & 1\\
 	%0 & \lambda_b\end{array}\right],
 	In order for $X=(A,B,I,J,A',B',F)$ to be a stable representation of the enhanced ADHM quiver, it is easy to check that
 	\begin{equation}
 	\label{eq:X-n-diag-03}
 	I =\left[\begin{array}{c}
 	\mu\\
 	\lambda\\
 	1
 	\end{array}\right].
 	\end{equation}
 	\indent Now, consider $v\in T_X\mathcal{N}$ given by $v=(a,b,i,j,a',b',f)$. Then, it follows from \eqref{eq:tangent_space} that $v$ satisfies
 	\begin{align*}
 	j=0,\quad [a,B] + [A,b]=0,\quad fA'+Fa'-aF-Af=0,\quad fB'+Fb'-bF-Bf=0.
 	\end{align*}
 	Then, denoting 
 	\begin{equation*}
 	a=\left[\begin{array}{ccc}
 	a_{11} & a_{12} & a_{13}\\
 	a_{21} & a_{22} & a_{23}\\
 	a_{31} & a_{32} & a_{33} \end{array}\right],\quad b=\left[\begin{array}{ccc}
 	b_{11} & b_{12} & b_{13}\\
 	b_{21} & b_{22} & b_{23}\\
 	b_{31} & b_{32} & b_{33} \end{array}\right], \quad i=\left[\begin{array}{c}
 	i_{1} \\
 	i_{2}\\
 	i_{3}\end{array}\right], \quad f=\left[\begin{array}{c}
 	f_{1} \\
 	f_{2} \\
 	f_{3}\end{array}\right]
 	\end{equation*}
 	one gets from the equations above that
 	 $v=(a,b,i,j,a',b',f)\in T_X\mathcal{N}$ is such that
 	\begin{align*}
 	a=\left[\begin{array}{ccc}
a' & a_{12} & a_{13}\\
(A-A_2)f_2 & a_{22} & a_{23}\\
(A-A_3)f_3 & a_{32} & a_{33} \end{array}\right],\quad b=\left[\begin{array}{ccc}
b & b_{12} & b_{13}\\
(B-B_2)f_2 & b_{22} & b_{23}\\
(B-B_3)f_3 & b_{32} & b_{33} \end{array}\right], \quad i=\left[\begin{array}{c}
i_{1} \\
i_{2}\\
i_{3}\end{array}\right], \quad f=\left[\begin{array}{c}
f_{1} \\
f_{2} \\
f_{3}\end{array}\right]
 	\end{align*}
 	and satisfies 
 	\begin{align}
 	\left(A-A_2\right) b_{12} = a_{12} \left(B-B_2\right) \label{eq:diag1} \\
 	 \left(A-A_3\right) b_{13} = a_{13} \left(B-B_3\right) \label{eq:diag2}\\
\left(A_2-A_3\right) b_{23} = a_{23} \left(B_2-B_3\right) \label{eq:diag3}\\
\left(A_2-A_3\right) b_{32} = a_{32} \left(B_2-B_3\right)\label{eq:diag4}  
\end{align}
 	\indent Thus, for $u_1=\left(a_1,b_1,i_1,j_1,a_1',b_1',f_1\right)$ $u_2=\left(a_2,b_2,i_2,j_2,a_2',b_2',f_2\right)$, such that
 	\begin{align*}
a_1= \left(
\begin{array}{ccc}
	a_1' & a1_{12} & a1_{13} \\
	\left(A-A_2\right) f2_2 & a1_{22} & a1_{23} \\
	\left(A-A_3\right)f2_3 & a1_{32} & a1_{33} \\
\end{array}
\right)\\
b_1= \left(
\begin{array}{ccc}
	b_1' & b1_{12} & b1_{13} \\
	\left(B-B_2\right) f2_2 & b1_{22} & b1_{23} \\
	\left(B-B_3\right) f2_3 & b1_{32} & b1_{33} \\
\end{array}
\right)\\
a_2= \left(
\begin{array}{ccc}
	a_2' & a2_{12} & a2_{13} \\
	\left(A-A_2\right) f2_2 & a2_{22} & a2_{23} \\
	\left(A-A_3\right)f2_3 & a2_{32} & a2_{33} \\
\end{array}
\right)\\
b_2=\left(
\begin{array}{ccc}
	b_2' & b2_{12} & b2_{13} \\
	\left(B-B_2\right) f2_2 & b2_{22} & b2_{23} \\
	\left(B-B_3\right) f2_3 & b2_{32} & b2_{33} \\
\end{array}
\right)\\
f1=\left(
\begin{array}{c}
	f1_1 \\
	f1_2 \\
	f1_3 \\
\end{array}
\right) 
f2=\left(
\begin{array}{c}
	f2_1 \\
	f2_2 \\
	f2_3 \\
\end{array}
\right), 
\end{align*}
the holomorphic pre-symplectic for $\Omega_X(u_1,u_2)$ is given by
\begin{align}
\Omega_X(u_1,u_2) & =  tr(-a_2b_1+b_2a_1)-a_2'b_1'+b_2'a_1' \label{eq:form diag}\\
                  & = -2 a_2' b_1'+2 a_1' b_2'-a2_{22} b1_{22}-a2_{32} b1_{23}-a2_{23} b1_{32}-a2_{33}
                  b1_{33}+a1_{22} b2_{22}+a1_{32} b2_{23} \nonumber \\ 
                  & +a1_{23} b2_{32}+a1_{33} b2_{33}+a2_{12}
                  \left(-B+B_2\right) f1_2+\left(A-A_2\right) b2_{12} f1_2+a2_{13} \left(-B+B_3\right) f1_3 \nonumber\\
                  & +\left(A-A_3\right) b2_{13}
                  f1_3+a1_{12} \left(B-B_2\right) f2_2+\left(-A+A_2\right) b1_{12} f2_2+ \nonumber\\
                  &(a1_{13} \left(B-B_3\right) +\left(-A+A_3\right)
                  b1_{13}) f2_3 \nonumber\\
                  & = 2 a_2' b_1'+2 a_1' b_2'-a2_{22} b1_{22}-a2_{32} b1_{23}-a2_{23} b1_{32}-a2_{33}
                  b1_{33}+ \nonumber\\
           & a1_{22} b2_{22}+a1_{32} b2_{23}+a1_{23} b2_{32}+a1_{33} b2_{33}+b2_{12} \left(-A+A_2\right)
              f1_2+ \nonumber\\
         &\left(A-A_2\right) b2_{12} f1_2+b2_{13} \left(-A+A_3\right) f1_3+\left(A-A_3\right) b2_{13} f1_3+b1_{12}
          \left(A-A_2\right) f2_2+ \nonumber\\
&  \left(-A+A_2\right) b1_{12} f2_2+b1_{13} \left(A-A_3\right) f2_3+\left(-A+A_3\right) b1_{13} f2_3 \nonumber \\
& = 2 a_2' b_1'+2 a_1' b_2'-a2_{22} b1_{22}-a2_{32} b1_{23}-a2_{23} b1_{32}-a2_{33}
b1_{33}+a1_{22} b2_{22}+a1_{32} \nonumber \\
& b2_{23}+a1_{23} b2_{32}+a1_{33} b2_{33} \nonumber
\end{align}
where in the third equality we used equations \eqref{eq:diag1}, \eqref{eq:diag2}, \eqref{eq:diag3}, \eqref{eq:diag4}. 
 	Suppose that $\Omega_{X}(u_1,u_2)=0$, for all $u_2$ on the tangent space. In particular, taking 
 	 $u_1=\left(a_1,b_1,i_1,j_1,0,0,f_1\right)$ such that
 	\begin{align*}
 	a_1= \left(
 	\begin{array}{ccc}
 	0 & a1_{12} & a1_{13} \\
 	\left(A-A_2\right) f2_2 & 0 & 0 \\
 	\left(A-A_3\right)f2_3 & 0 & a1_{33} \\
 	\end{array}
 	\right)\\
 	b_1= \left(
 	\begin{array}{ccc}
 	0 & b1_{12} & b1_{13} \\
 	\left(B-B_2\right) f2_2 & 0 & 0 \\
 	\left(B-B_3\right) f2_3 & 0 & 0 \\
 	\end{array}
 	\right)\\
 	f1=\left(
 	\begin{array}{c}
 	f1_1 \\
 	f1_2 \\
 	f1_3 \\
 	\end{array}
 	\right), 
 	\end{align*}
 	
 	it follows from equation \eqref{eq:form diag} that $$\Omega_{X}(u_1,u_2) = a1_{33}b2_{33} = 0,$$ which means that $a1_{33} = 0$. Analogously one can check that if $\Omega_{X}(u_1,u_2)=0$, for all $u_2$ on the tangent space,  $u_1=\left(a_1,b_1,i_1,j_1,a_1',b_1',f_1\right)$ must be such that
 	\begin{align*}
 	a_1= \left(
 	\begin{array}{ccc}
 	0 & a1_{12} & a1_{13} \\
 	\left(A-A_2\right) f2_2 & 0 & 0 \\
 	\left(A-A_3\right)f2_3 & 0 & 0 \\
 	\end{array}
 	\right)\\
 	b_1= \left(
 	\begin{array}{ccc}
 	0 & b1_{12} & b1_{13} \\
 	\left(B-B_2\right) f2_2 & 0 & 0 \\
 	\left(B-B_3\right) f2_3 & 0 & 0 \\
 	\end{array}
 	\right)\\
 	f1=\left(
 	\begin{array}{c}
 	f1_1 \\
 	f1_2 \\
 	f1_3 \\
 	\end{array}
 	\right),\quad 
 	a_1'= b_1' = 0. 
 	\end{align*}
 	Denote by $[u_1]$ the equivalence class of $u_1$. In order to conclude this case, we must prove that $[u]=[0]$ i.e., $u\in Im(d_0)$, where
 	\begin{equation*}
 	\begin{array}{lcll}
 	do:& End(V)\oplus End(V') & \longrightarrow & \mathbb{X}\\
 	& (h,h')               & \longmapsto     & ([h,A], [h,B], hI, -Jh, 0, 0, hF-Fh')
 	\end{array}.
 	\end{equation*}
 	
 	However, by means of tedious computation, one can check that for 
 \begin{align*}
 	H = \left(
 	\begin{array}{ccc}
 		\frac{x-H_{12} i_2-H_{13} i_3}{i_1} & H_{12} & H_{13} \\
 		f1_2 & \frac{y-f1_2 i_1}{i_2} & 0 \\
 		f1_3 & 0 & \frac{y-f1_2 i_1}{i_2} \\
 	\end{array}
 	\right), \quad h=\frac{x-f1_1 i_1-H_{12} i_2-H_{13} i_3}{i_1}
\end{align*} 
$d_0(H,h) = u_1$, concluding the proof.	
 \end{proof}


\begin{thebibliography}{99999}
\bibliographystyle{plain}

\bibitem{B}
V. Baranovsky,
Moduli of sheaves on surfaces and action of the oscillator algebra.
J. Differential Geom. {\bf 55} (2000), 193--227.

\bibitem{A01-01}
U. Bruzzo, W.-Y. Chuang, D.-E. Diaconescu, M. Jardim, G. Pan, Y. Zhang, 
D-branes, surface operators, and ADHM quiver representations.
Adv. Theor. Math. Phys. {\bf 15} (2011), 849--911.

\bibitem{BM}
U. Bruzzo, D. Markushevich,
Moduli of framed sheaves on projective surfaces.
Documenta Math. {\bf 16} (2011), 399--410.

\bibitem{BE}
M. Bulois, L. Evain,
Nested punctual Hilbert schemes and commuting varieties of parabolic subalgebras.
J. Lie Theory {\bf 26} 2016, 497--533.

\bibitem{C}
J. Cheah,
Cellular decompositions for nested Hilbert schemes of points.
Pacific J. Math, {\bf 183} (1998), 39--90.

\bibitem{CDDT}
W.-Y. Chuang, D.-E. Diaconescu, R. Donagi, T. Pantev,
Parabolic refined invariants and Macdonald polynomials.
Commun. Math. Phys. {\bf 335} (2015), 1323--1379.

\bibitem{G}
I. Grojnowski,
Instantons and Affine Algebras I. The Hilbert Scheme and Vertex Operators.
Math. Res. Lett. {\bf 3} (1996), 275-291.

\bibitem{gross}
M. Gross, D. Huybrechts, D. Joyce,
Calabi-Yau manifolds and related geometries.
Lectures from the Summer School held in Nordfjordeid, June 2001. 
Spring-Verlag, Berlin, 2003.

\bibitem{HL}
D. Huybrechts, M. Lehn,
Stable pairs on curves and surfaces.
J. Alg. Geom. {\bf 4} (1995), pp. 67--104.

\bibitem{A01-07}
M. Jardim, R. V. Martins, 
The ADHM variety and perverse coherent sheaves.
J. Geom. Phys. {\bf 65} (2011), 2219--2232.

\bibitem{vFJ}
R. A. von Flach, M. Jardim,
Moduli Spaces of Framed Flags of Sheaves on the Projective Plane
J. Geom. Phys. {\bf 118} (2017), 138--168.

\bibitem{A01-01-20}
A. King, 
Moduli of representations of finite dimensional algebras.
Q. J. Math. {\bf 45} (1994) 515--530.

\bibitem{N1}
H. Nakajima,
Quiver varieties and Kac--Moody algebras. 
Duke Math. J. {\bf 91} (1998), 515--560.

\bibitem{nakajima}
H. Nakajima, 
Lectures on Hilbert schemes of points on surfaces. 
American Mathematical Society, Providence, RI, 1999.

\bibitem{N2}
H. Nakajima,
Lectures at the University of Hong Kong -- a geometric construction of algebras.
Available at 
\begin{verbatim}
http://www.kurims.kyoto-u.ac.jp/~nakajima/TeX/HongKong/hongkong.pdf
\end{verbatim}
Last access: 13 Dec 2016.

\bibitem{okonek}
C. Okonek, M. Schneider, H. Spindler, Vector bundles on complex projective spaces. 
Progress in Mathematics, Birakhauser, 1980.

\bibitem{patricia}
P. B. dos Santos, 
ADHM construction of nested Hilbert schemes.
Ph.D. thesis, Universidade Estadual de Campinas, Campinas, Brazil, 2016.

\bibitem{T}
A. S. Tikhomirov,
The variety of complete pairs of zero-dimensional subschemes of an algebraic surface.
Izvestiya: Mathematics {\bf 61} (1997), 153--180.

\end{thebibliography}
\end{document}